\documentclass[12pt]{article}
\usepackage[utf8]{inputenc}
\usepackage{scrextend}
\usepackage{amsmath}
\usepackage{graphicx}
\usepackage{amssymb}
\usepackage{amsthm}
\usepackage{enumerate}
\usepackage{caption}
\usepackage{subcaption}
\usepackage{hyperref}
\usepackage{float}
\usepackage{cite}
\usepackage[left=3cm, right=3cm,top=3cm,bottom=3cm]{geometry}
\newtheorem{theorem}{Theorem}
  \newtheorem{lemma}[theorem]{Lemma}
  
  \newtheorem{conjecture}[theorem]{Conjecture}
  
  \newtheorem{proposition}[theorem]{Proposition}

\newcommand{\C}{\mathcal{C}}
\newcommand{\G}{\mathcal{G}}
\newcommand{\N}{\mathbb{N}}

\title{Gaps in the cycle spectrum of 3-connected cubic planar graphs}
\author{Martin Merker\footnote{Department of Applied Mathematics and Computer Science, Technical University of Denmark, DK-2800 Lyngby, Denmark. E-mail address: marmer@dtu.dk. The author was supported by the Danish Council for Independent Research, Natural Sciences, grant DFF-8021-00249, AlgoGraph.}}
\date{}

\begin{document}

\maketitle

\begin{abstract}
We prove that, for every natural number $k$, every sufficiently large 3-connected cubic planar graph has a cycle whose length is in $[k,2k+9]$. We also show that this bound is close to being optimal by constructing, for every even $k\geq 4$, an infinite family of 3-connected cubic planar graphs that contain no cycle whose length is in $[k,2k+1]$.
\end{abstract}

\section{Introduction}

The \emph{cycle spectrum} of a graph $G$, denoted $\C (G)$, is the set of lengths of cycles in $G$. The \emph{circumference} of a graph is the length of its longest cycle. The size of the cycle spectrum has been studied for many different graph classes, in particular for graphs of large minimum degree and Hamiltonian graphs. For example, Sudakov and Verstraëte~\cite{sudakov} showed that if $G$ has average degree $d$ and girth $g$, then $|\C (G)| \geq \Omega (d^{\lfloor \frac{g-1}{2}\rfloor})$, proving a conjecture by Erd\H{o}s~\cite{erdos2}. Milans et al.~\cite{milans} showed that if $G$ is Hamiltonian, then $|\C (G)| \geq \sqrt{p} - \frac{1}{2}\ln{p}-1$ where $p= |E(G)|-|V(G)|$.\\
There are various problems in structural graph theory which can be rephrased in terms of the cycle spectrum of graphs with certain properties. A typical question is to decide for a class of graphs $\G$ and $A\subset \N$ if every graph in $\G$ contains a cycle whose length is in~$A$. For example, Grötzsch's theorem~\cite{grotzsch} states that $3\in \C (G)$ for every planar graph~$G$ of chromatic number 4, while Steinberg's Conjecture states that $\C (G) \cap \{4,5\} \neq \emptyset$. Steinberg's Conjecture has been disproved~\cite{cohen-addad} but it is still open if $\C (G) \cap \{4,5,6\} \neq \emptyset$ for every planar graph $G$ of chromatic number 4. There are also many theorems and open problems where $A$ is an infinite set. For example, the Erd\H{o}s-Gy{\'a}rf{\'a}s Conjecture~\cite{erdos} states that if $G$ is a graph of minimum degree at least~3, then $\C (G)\cap \{2^k:k\in \N\} \neq \emptyset$. For $m,k\in \N$, let $A(m,k)$ denote the set of natural numbers congruent to $m$ modulo~$k$. Very recently, Gao et. al~\cite{gao} showed that $\C (G)\cap A(2m,k) \neq \emptyset$ for every graph $G$ of minimum degree at least $k+1$, proving a conjecture by Thomassen~\cite{thomassen}. Lyngsie and Merker~\cite{lyngsie} showed that $\C (G)\cap A(m,2k+1) \neq \emptyset$ if $G$ is a sufficiently large 3-connected cubic graph.\\
In this article we focus on the case where $A$ is an interval and $\G$ is a family of 2-connected simple graphs. Let $a,b\in \N$ with $3\leq a \leq b$. We say $[a,b]$ is a \emph{gap} of $G$ if $\C(G) \cap [a,b] = \emptyset$ and $G$ has circumference greater than $b$. Note that if $[a,b]$ is a gap of $G$ then also every interval contained in $[a,b]$ is a gap of $G$. We investigate which intervals occur as gaps of graphs in $\G$. If $\G$ contains graphs of arbitrarily large girth, then every interval is a gap of some graph in $\G$. To avoid this, we focus on 2-connected cubic planar graphs.

\begin{proposition}\label{prop}
Let $a,b\in \N$ with $3\leq a\leq b$. The interval $[a,b]$ is a gap of some 2-connected cubic planar graph if and only if $a=3$, $b\leq 4$ or $a=4$, $b\leq 9$ or $a\geq 5$.
\end{proposition}
\begin{proof}
Every 2-connected cubic plane graph contains a face of length 3, 4, or 5.
Thus, if $[3,b]$ is a gap of some 2-connected cubic planar graph, then $b\in \{3,4\}$. The interval $[3,4]$ is a gap of every cubic planar graph of girth 5. If $G$ is a cubic planar graph of girth 5, then replacing every vertex of $G$ with a triangle yields a graph $G'$ with $\C(G')\cap [4,9] = \emptyset$. Thus, $[4,9]$ is a gap of some 2-connected cubic planar graph. A simple discharging argument shows that $\C (G) \cap [4,10] \neq \emptyset$ for every 2-connected cubic planar graph $G$.\\
Now assume $a\geq 5$ and let $C = v_1v_2\ldots v_{3k}v_1$ be a cycle on $3k$ vertices where $3k>b$. Let $G$ be the graph consisting of~$C$ and $k$ vertices $u_1,\ldots ,u_k$ such that $u_i$ is joined to $v_{3i-2}$, $v_{3i-1}$, and $v_{3i}$ for $i\in \{1,\ldots ,k\}$. It is easy to see that $G$ is a Hamiltonian 2-connected cubic planar graph and the only cycles of length less than $3k$ have length 3 or 4. Thus, $\C (G) \cap [a,b] = \emptyset$ and $[a,b]$ is a gap of $G$.
\end{proof}

Proposition~\ref{prop} completely characterizes the intervals which are gaps of 2-connected cubic planar graphs. 
For 3-connected cubic planar graphs the characterization of gaps appears to be more difficult. In Section 3 we construct 3-connected cubic planar graphs for which $[k,2k+1]$ is a gap (with $k$ even). The main theorem of this paper shows that gaps in 3-connected cubic planar graphs cannot be much larger than that.

\begin{theorem}\label{thm:main}
If $k\in \N$ and $G$ is a 3-connected cubic planar graph of circumference at least $k$, then $\C (G) \cap [k,2k+9]\neq \emptyset$.
\end{theorem}

We make the following conjecture which states that the graphs constructed in Section 3 have the largest possible gaps among all 3-connected cubic planar graphs.

\begin{conjecture}
If $k\in \N$ with $k\geq 2$ and $G$ is a 3-connected cubic planar graph of circumference at least $k$, then $ \C (G) \cap [k,2k+2]\neq \emptyset$.
\end{conjecture}

It is not clear whether 3-connected planar graphs can have larger gaps in their cycle spectrum compared to 3-connected cubic planar graphs. We conjecture that a similar result to Theorem~\ref{thm:main} holds for this more general class of graphs.

\begin{conjecture}
There exists $c\in \N$ such that $ \C (G) \cap [k,2k+c]\neq \emptyset$ for every $k\in \N$ and 3-connected planar graph $G$ of circumference at least $k$.
\end{conjecture}

\section{Proof of Theorem~\ref{thm:main}}

We prove Theorem~\ref{thm:main} by contradiction. Suppose there exists a 3-connected cubic plane graph $G$ of circumference at least $k$ such that $\C(G)\cap [k,2k+9] = \emptyset$. Now every cycle in $G$ is either \emph{short} (shorter than $k$) or \emph{long} (longer than $2k+9$). 
Since the circumference of $G$ is at least $k$, there is at least one long cycle in $G$. The following lemma shows that $G$ also contains a long facial cycle.

\begin{lemma}\label{lem:longface}
Let $G$ be a 3-connected cubic plane graph and $k,c\in \N$. If the circumference of $G$ is at least $k$, then $\C (G)\cap [k,2k+c] \neq \emptyset$ or $G$ has a facial cycle of length greater than $2k+c$.
\end{lemma}
\begin{proof}
Suppose $\C (G)\cap [k,2k+c] = \emptyset$ and every facial cycle of $G$ has length less than $k$. Let $C$ be a cycle of length greater than $2k+c$ for which the number of faces in its interior is minimal. Since every face has length less than $k$, there are at least three faces in the interior of $C$. For every edge $e$ of $C$, let $C_e$ denote the facial cycle in the interior of $C$ which is incident with $e$. Let $D_e$ denote the symmetric difference of $C$ and $C_e$. Note that $D_e$ is a union of cycles for every $e\in E(C)$. It is easy to see that there exists an edge $f\in E(C)$ such that $D_f$ is a cycle.
Since every facial cycle has length less than $k$, we have $|E(D_f)|\geq |E(C)|-k$. Since $|E(C)|> 2k+c$ and $\C (G)\cap [k,2k+c] = \emptyset$, we have $|E(D_f)|>2k+c$. However, $D_f$ contains fewer faces than $C$ in its interior, contradicting our choice of $C$.
\end{proof}

Next we show that we can remove edges between short faces of $G$ to obtain a 2-connected subgraph $G'$ where no two short faces are adjacent. Moreover, the subgraph $G'$ also contains a long facial cycle and each edge of a 2-edge-cut in $G'$ is incident with both a short face and a long face.

\begin{lemma}\label{lem:reduction}
Let $G$ be a 3-connected cubic plane graph and $k\in \N$. If $\C (G) \cap [k,2k]= \emptyset$ and $G$ contains a facial cycle $C$ of length at least $2k+1$, then $G$ contains a 2-connected subgraph $G'$ such that
\begin{description}
    \item[(A)] no two facial cycles of $G'$ of length less than $k$ intersect,
    \item[(B)] $E(C)\subseteq E(G')$,
    \item[(C)] and every edge of $G'$ which is part of a 2-edge-cut is incident with a face longer than $2k$ and a face shorter than $k$.
\end{description}
\end{lemma}
\begin{proof}
We call a cycle \emph{short} if it has length less than $k$ and \emph{long} if it has length greater than $2k$. We call a face short (long) if its boundary is a short (long) cycle. Since $\C (G) \cap [k,2k]= \emptyset$, every cycle of $G$ is short or long.\\
We construct a sequence of 2-connected graphs $G_0,G_1,\ldots $ by successively gluing adjacent short faces together. Let $G_0=G$ and suppose we have constructed $G_i$. If $G_i$ does not contain two adjacent short faces, then we set $G'=G_i$ and stop. Otherwise, let $F_1$ and $F_2$ be two adjacent short faces in $G_i$. We form a graph $H_i$ by deleting the edges of $F_1\cap F_2$ and we also delete any resulting isolated vertices. Now $H_i$ has a face~$F$ whose boundary is the symmetric difference of $F_1$ and $F_2$. This boundary might consist of several cycles, but each cycle in the boundary is short. This operation cannot create cut-edges, but it can disconnect the graph, in which case $F$ is incident with each component. If $H_i$ is disconnected, then we choose the component containing $C$ as $G_{i+1}$, otherwise we set $G_{i+1}=H_i$. Note that every long face of $G_{i+1}$ is also a long face of $G_i$.
Clearly $G_{i+1}$ is a proper subgraph of $G_i$, so the sequence terminates.\\
In $G'$ there are no two adjacent short faces, so $G'$ satisfies (A). We have $E(C)\subset E(G_i)$ for each $i$, so $G'$ satisfies (B). Every $G_i$ is 2-connected, so $G'$ is 2-connected. If $e_1,e_2$ is a 2-edge-cut in $G'$, then there are two faces $F_1$ and $F_2$ which are incident with both $e_1$ and $e_2$. Since $G'$ satisfies (A), at least one of these two faces is long. If both $F_1$ and $F_2$ are long, then $F_1$, $F_2$ are also faces of $G$ and $e_1,e_2$ would be a 2-edge-cut in $G$, contradicting 3-connectivity. Thus, every edge of a 2-edge-cut is incident with a short face and a long face, so $G'$ satisfies (C).
\end{proof}

Note that (C) implies that if $C_1$ and $C_2$ are two facial cycles in $G'$ with length greater than $k$, then $|E(C_1)\cap E(C_2)| \leq 1$. We can now prove Theorem~\ref{thm:main} by a counting argument.

\begin{proof}[Proof of Theorem~\ref{thm:main}]
We fix an embedding of $G$ in the plane. Suppose the circumference of $G$ is at least $k$ and $\C (G) \cap [k,2k+9] = \emptyset$.
By Lemma~\ref{lem:longface} we can assume that $G$ has a facial cycle~$C$ with $|E(C)|\geq 2k+10$. We apply Lemma~\ref{lem:reduction} to obtain a 2-connected subgraph $G'$ of $G$ satisfying the conditions (A), (B), and (C). Let $H$ be the graph we obtain from $G'$ by suppressing all vertices of degree 2. The embedding of $G$ in the plane also gives us an embedding of $G'$ and $H$ such that there is a canonical bijection between the faces of $G'$ and $H$. Let $F(G')$ denote the set of faces of $G'$. For $F\in F(G')$ we write $\ell(F)$ for the length of $F$ and $\ell_H(F)$ for the length of the corresponding face in $H$. We define $X = \{F\in F(G'): \ell(F) < k\}$ and $Y = \{F\in F(G'): \ell(F) > 2k+9\}$. Note that $F(G') = X\cup Y$. Let $x = |X|$, $y = |Y|$, and $n=|V(H)|$. 
The graph $H$ is a 2-connected cubic planar graph, so
\begin{align}\label{eq:euler}
    x+y = \frac{n}{2}+2
\end{align}
by Euler's formula. \\
Suppose $H$ contains a facial cycle of length 2. Let $C'$ be the corresponding facial cycle in $G'$, and let $C_1$, $C_2$ denote the two facial cycles in $G'$ that intersect $C'$. Since there exists a 2-edge-cut incident with $C_1$ and $C_2$, by (C) we may assume $|E(C_1)|<k$ and $|E(C_2)| > 2k+9$. By (A), we have $|E(C')| > 2k+9$. Now 
$$2k+9<|E(C')| = |E(C'\cap C_1)| +  |E(C'\cap C_2)| < k+|E(C'\cap C_2)|$$ 
implies $|E(C'\cap C_2)|>k+9$ which contradicts (C). Thus, every facial cycle of $H$ has length at least 3. By (A) the faces in $X$ are pairwise non-adjacent, thus
\begin{align}\label{eq:xn}
   n = |V(H)| \geq \sum_{F\in X}\ell_H(F) \geq 3x
\end{align}
which implies
\begin{align}\label{eq:longinH}
    \sum_{F\in Y}\ell_H(F)  = \sum_{F\in F(G')}\ell_H(F) - \sum_{F\in X}\ell_H(F)
     = 2|E(H)| - \sum_{F\in X}\ell_H(F) 
     \leq 3n-3x\,.
\end{align}
In $G'$, each edge incident with a face in $Y$ is also an edge in $H$ or it is incident with a face in $X$, so by (\ref{eq:longinH}) we have
\begin{align}\label{eq:upper_bound}
    \sum_{F\in Y}\ell(F) & \leq \sum_{F\in Y}\ell_H(F) + \sum_{F\in X}\ell(F) \leq 3(n-x) + (k-1)x = 3n + (k-4)x\,.
\end{align}
Each face in $Y$ has length at least $2k+10$ in $G'$, so by (\ref{eq:euler}) we have
\begin{align}\label{eq:lower_bound}
    \sum_{F\in Y}\ell(F) \geq (2k+10)y = (2k+10)\left(\frac{n}{2}+2-x\right) > (k+5)n-(2k+10)x\,.
\end{align}
Combining (\ref{eq:upper_bound}) and (\ref{eq:lower_bound}) yields
\begin{align*}
    (k+5)n -(2k+10)x < 3n+(k-4)x
\end{align*}
which is equivalent to
\begin{align*}
    (k+2)n < 3(k+2)x\,.
\end{align*}
This implies $n < 3x$ which contradicts (\ref{eq:xn}).
\end{proof}

\section{Graphs with large cycle gaps}

In this section we show that for every $k\geq 2$ the interval $[2k,4k+1]$ is a gap of a 3-connected cubic planar graph and thus the interval $[k,2k+9]$ in Theorem~\ref{thm:main} is close to the optimal bound.

\begin{theorem}
For $k\in \N$ with $k\geq 2$ there exists a 3-connected cubic planar graph $G$ of circumference at least $2k$ such that $\C (G) \cap [2k,4k+1] = \emptyset$.
\end{theorem}
\begin{proof}
For $k\in \N$, let $H_k$ be the graph consisting of two disjoint cycles $C_1 = x_0x_1\ldots x_{2k}x_0$ and $C_2 = y_0y_1\ldots y_{2k}y_0$ together with chords $x_ix_{2k-i}$ and $y_iy_{2k-i}$ for $i\in \{1,\ldots ,k-1\}$ as well as one edge $x_ky_k$ joining $C_1$ and $C_2$, see Figure~\ref{fig:H_4}.\\
Let $D$ be a cubic graph and $e\in E(D)$. We define $D(e,H_k)$ as the graph obtained from $D$ by replacing $e$ with a copy of $H_k$. To be more precise, if $e=uv$ where $u$ is adjacent to $u_1,u_2,v$ and $v$ is adjacent to $v_1,v_2,u$ in $D$, then $D(e,H_k)$ is obtained from the disjoint union of $D-u-v$ and $H_k$ by adding the edges $u_1x_0$, $u_2x_{2k}$, $v_1y_0$, and $v_2y_{2k}$. If $M$ is a matching in $D$, then we define $D(M,H_k)$ as the graph we obtain by successively replacing the edges in $M$ by $H_k$. Note that if $D$ is a 3-connected cubic planar graph, then so is $D(M,H_k)$. If $M$ is a perfect matching, then $D(M,H_1)$ is the graph where every vertex of $D$ is replaced with a triangle.\\
For $n\in \N$ with $n\geq 4k+2$, let $D_n$ be a graph consisting of two cycles $C_1 = u_1\ldots u_nu_1$, $C_2 = v_1\ldots v_nv_1$ of length $n$ and one cycle $C_3 = w_1w_2\ldots w_{2n}w_1$ of length $2n$ together with edges $v_iw_{2i-1}$ and $u_iw_{2i}$ for every $i\in \{1,\ldots ,n\}$. It is easy to see that $D_n$ is a 3-connected cubic planar graph. Let $M$ denote the perfect matching in $D$ consisting of the edges of the form $v_iw_{2i-1}$ and $u_iw_{2i}$ for $i\in \{1,\ldots ,n\}$. We define $G(n,k)=D_n(M,H_{k-1})$, see Figure~\ref{fig:G(4,3)} for an example. Clearly $G(n,k)$ is a 3-connected cubic planar graph. Since $n\geq 4k+2$, the circumference of $G(n,k)$ is greater than $2k$. It remains to show that $\C(G(n,k))\cap [2k,4k+1] =\emptyset$.\\
Let $C$ be a cycle of $G(n,k)$ and suppose $2k\leq |E(C)| \leq 4k+1$. The circumference of $H_{k-1}$ is $2k-1$, so $C$ is not contained in one of the copies of $H_{k-1}$. Thus, $C$ corresponds to a cycle $C'$ in $D$ (by contracting each block of each copy of $H_{k-1}$ into a vertex) and $|E(C')|\leq |E(C)|$. The shortest cycle in $D$ not containing an edge of $M$ has length $n$, so $C'$ contains at least one edge of $M$. There are no cycles in $D$ containing precisely one edge of $M$, so $|E(C')\cap M|\geq 2$. Note that in $C$ the edges used by $C'$ in $M$ correspond to paths between vertices $x_0$ and $y_0$ in $H_{k-1}$ and these have length at least $2k-1$. In particular $|E(C')\cap M| = 2$. Moreover, $C'$ contains two consecutive edges not belonging to $M$ and these two edges correspond to a path of length 3 in $C$. Since the girth of $D$ is 5, we have $|E(C)|\geq |E(C')|+2(2k-2)+1 \geq 5+4k-3 = 4k+2$, contradicting our choice of $C$.
\end{proof}

\begin{figure}[t]
    \centering
    \includegraphics[scale = 0.25]{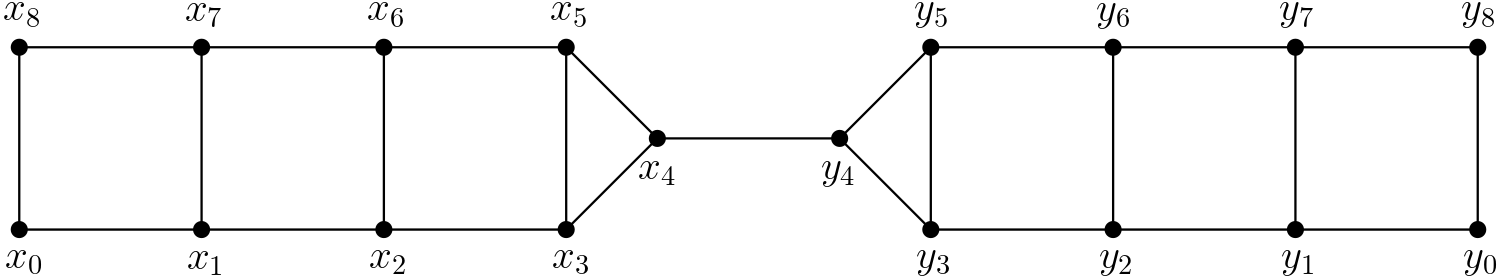}
    \caption{The graph $H_4$}
    \label{fig:H_4}
\end{figure}

\begin{figure}[h]
    \centering
    \includegraphics[scale = 0.45]{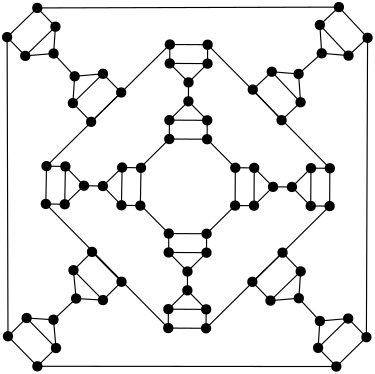}
    \caption{The graph $G(4,3)$}
    \label{fig:G(4,3)}
\end{figure}

\end{document}